\documentclass{article}[twosided]
\usepackage{comment}
\usepackage{graphicx}  
\usepackage{float}      
\usepackage{caption}
\usepackage{soul}

\newif\ifaistats
\newif\ifarxiv
\newif\ifarxivtwocolumns
\newif\ifneurips
\newif\ificml
\newif\ifaistatsnot
\newif\ifarxivnot
\newif\ifarxivtwocolumnsnot
\newif\ifneuripsnot
\newif\ificmlnot

\newif\ifextension
\newif\ifextensionnot
\extensionfalse
\extensionnottrue
\newif\ifhidden
\newif\ifhiddennot
\hiddenfalse
\hiddennottrue

\aistatsfalse
\arxivfalse
\arxivtwocolumnsnotfalse
\neuripsfalse
\icmlfalse
\aistatsnottrue
\arxivnottrue
\neuripsnottrue
\icmlnottrue

\ifaistats \aistatsnotfalse \fi
\ifarxiv \arxivnotfalse \fi
\ifarxivtwocolumnsnot \ifarxivtwocolumnsnotfalse \fi
\ifneurips \neuripsnotfalse \fi
\ificml \icmlnotfalse \fi
\ifextension \extensionnotfalse \fi
\ifhidden \hiddennotfalse \fi
\arxivtrue

\newcommand{\mbzuai}{\texttt{MBZUAI}}
\newcommand{\mytitle}{Loss-Transformation Invariance in the Damped Newton Method}
\newcommand{\weauthors}{\makecell{Alexander Shestakov \\ \mbzuai{}\footremember{MBZUAI} {Mohammed bin Zayed University of Artificial Inteligence. Email: firstname.lastname@mbzuai.ac.ae}}
\and \, \makecell{Sushil Bohara \\ \mbzuai{}\footrecall{MBZUAI}}
\and \, \makecell{Samuel Horv\'ath \\ \mbzuai{}\footrecall{MBZUAI}} 
\and \, \makecell{Martin Tak\'a\v{c} \\ \mbzuai{}\footrecall{MBZUAI}}
\and \, \makecell{Slavom\'ir Hanzely \\ \mbzuai{}\footrecall{MBZUAI}}
}

\ifaistats \aistatsnotfalse \fi
\ifarxiv \arxivnotfalse \fi
\ifarxivtwocolumns \arxivtwocolumnsnotfalse \fi
\ifneurips \neuripsnotfalse \fi
\ificml \icmlnotfalse \fi
\ifextension \extensionnotfalse \fi
\ifhidden \hiddennotfalse \fi

\usepackage[utf8]{inputenc} 
\usepackage[T1]{fontenc}    
\usepackage{hyperref}       
\usepackage{url}            
\usepackage{booktabs}       
\usepackage{amsfonts}       
\usepackage{nicefrac}       
\usepackage{microtype}      
\usepackage{xcolor}         
\usepackage{wrapfig}
\ifaistatsnot
    \usepackage{natbib}
\fi
\usepackage{xargs}                      
\usepackage{comment}

\usepackage{hyperref}
\usepackage{url, times}

\usepackage{amssymb,amsmath,amscd,amsfonts,amstext,amsthm,bbm, enumerate, dsfont, mathtools}
\usepackage{array}
\usepackage{multicol}
\usepackage{graphicx}
\usepackage{varwidth}
\usepackage{hyperref}
\usepackage{import}
\usepackage{caption,subcaption}
\ificmlnot
    \usepackage{algorithm}
    \usepackage{algpseudocode}
\fi
\usepackage[none]{hyphenat}
\usepackage{footnote}
\usepackage{cleveref}
\usepackage{makecell}
\usepackage{threeparttable}
\usepackage{booktabs}
\usepackage{tcolorbox}
\usepackage{caption}
\usepackage{subcaption}

\usepackage{scrextend}

\usepackage{tablefootnote}
\usepackage{changepage}
\usepackage{stackengine}
\usepackage{pifont}

\usepackage{mdframed}

\graphicspath{ {images/} }

\newtheorem{theorem}{Theorem}

\newtheorem{lemma}{Lemma}
\newtheorem{corollary}{Corollary}

\newtheorem{definition}{Definition}

\DeclareMathOperator*{\argmin}{argmin}
\DeclareMathOperator*{\argmax}{argmax}

\newcommand{\R}{\mathbb{R}}

\newcommand{\eqdef}{\stackrel{\text{def}}{=}}

\mdfdefinestyle{definition}{
	backgroundcolor=lightgray!20,
    hidealllines=true,
}
\mdfdefinestyle{theorem}{
    linecolor=blue!50,
}
\mdfdefinestyle{proof}{
    linecolor=orange,
}
\mdfdefinestyle{proposition}{
    backgroundcolor=orange!5,
    hidealllines=true,
}
\mdfdefinestyle{corollary}{
    linecolor=blue!50,
}
\mdfdefinestyle{lemma}{
    linecolor=blue!50,
}

\surroundwithmdframed[style=definition]{definition}
\surroundwithmdframed[style=definition]{defalign}
\surroundwithmdframed[style=theorem]{theorem}
\surroundwithmdframed[style=proof]{proof}
\surroundwithmdframed[style=proposition]{proposition}

\newcolumntype{?}{!{\vrule width 1pt}}
\usepackage{colortbl}

\definecolor{mydarkgreen}{RGB}{39,130,67}

\definecolor{mydarkred}{RGB}{192,47,25}

\newcommand{\xopt}{x^*}

\newcommand{\pof}[1]{Proof of #1.}

\newcommand{\h}{\nabla^2 f}
\newcommand{\g}{\nabla f}

\newcommand{\mI}{\mathbf I}
\newcommand{\mH}{\mathbf H}

\newcommand{\normM}[2]{{\left \| #1 \right\|}_{#2}}

\newcommand{\normsMd}[2]{{\left \| #1 \right\|}_{#2}^{*2}}

\usepackage{theoremref}

\usepackage[colorinlistoftodos,bordercolor=orange,backgroundcolor=orange!20,linecolor=orange,textsize=scriptsize]{todonotes}

\usepackage{enumitem}

\newcommand{\lr}{\left(} 
\newcommand{\rr}{\right)}
\newcommand{\ls}{\left[} 
\newcommand{\rs}{\right]}

\newcommand{\la}{\left\langle} 
\newcommand{\ra}{\right\rangle}

\newcommand{\footremember}[2]{%
    \footnote{#2}
    \newcounter{#1}
    \setcounter{#1}{\value{footnote}}%
}
\newcommand{\footrecall}[1]{%
    \footnotemark[\value{#1}]%
}

\newtheorem{claim}{Claim}
\newtheorem{example}{Example}
\usepackage{pifont}

\ifaistats
    \usepackage{\string~"./styles/aistats2024"}
 
    \usepackage[round]{natbib}

\fi

\ifarxiv
    \usepackage[verbose=true,letterpaper]{geometry}
    \AtBeginDocument{
    \newgeometry{
        textheight=9in,
        textwidth=5.5in,
        top=1in,
        headheight=12pt,
        headsep=25pt,
        footskip=30pt
    }
    }

    \author{\weauthors}
    \date{}
    \title{\mytitle}
\fi

\ifarxivtwocolumns
    \usepackage{\string~"./styles/arxivtwocolumns"}
    \author{\weauthors}
    \date{}
    \title{\mytitle}
\fi

\ifneurips
    \usepackage{\string~"./styles/neurips_2023"}
    \author{\weauthors}
    \date{}
    \title{\mytitle}
\fi

\ificml
    \usepackage{\string~"./styles/icml2023"}
    \author{\weauthors}
    \date{}
    \title{\mytitle}
    
    \ificml

    \fi
\fi

\begin{document}
\ifaistats
        
        \twocolumn[
        \aistatstitle{\mytitle}
        \aistatsauthor{ Author 1 \And Author 2 \And  Author 3 }
        \aistatsaddress{ Institution 1 \And  Institution 2 \And Institution 3 } ]
\fi
\ifaistatsnot
	   \maketitle
\fi

\newcommand{\mA}{\mathbf A}
\newcommand{\xnext}{x^+}
\newcommand{\hi}{\ls \h(x) \rs^{-1}}
\newcommand{\loss}{L}
\newcommand{\gloss}{\nabla \loss}
\newcommand{\hloss}{\nabla^2 \loss}
\newcommand{\trans}{\phi}
\newcommand{\gtrans}{\trans'}
\newcommand{\htrans}{\trans''}
\newcommand{\atrans}{\alpha_\trans}
\newcommand{\gr}{g}

\newcommand{\ltrans}{\lambda_\phi}
\newcommand{\gtx}{\gtrans_x}
\newcommand{\htx}{\htrans_x}

\newcommand{\Int}{\text{Interior}}
\newcommand{\Bd}{\text{Boundary}}
\newcommand{\Epi}{\text{Epigraph}}

\def\<#1,#2>{\left\langle #1,#2 \right\rangle}

\newcommand{\slavo}[1]{\todo[inline]{{\textbf{Slavo:} \emph{#1}}}}
\newcommand{\aleksander}[1]{\todo[inline]{{\textbf{Aleksander:} \emph{#1}}}}
\newcommand{\sushil}[1] {\todo[inline]{{\textbf{Sushil:} \emph{#1}}}}
\newcommand{\martin}[1]{\todo[inline]{{\textbf{Martin:} \emph{#1}}}}
\newcommand{\samo}[1]{\todo[inline]{{\textbf{Samuel:} \emph{#1}}}}

\begin{abstract}
The Newton method is a powerful optimization algorithm, valued for its rapid local convergence and elegant geometric properties. However, its theoretical guarantees are usually limited to convex problems. In this work, we ask whether convexity is truly necessary. We introduce the concept of loss-transformation invariance, showing that damped Newton methods are unaffected by monotone transformations of the loss—apart from a simple rescaling of the step size. This insight allows difficult losses to be replaced with easier transformed versions, enabling convexification of many nonconvex problems while preserving the same sequence of iterates. Our analysis also explains the effectiveness of unconventional stepsizes in Newton’s method, including values greater than one and even negative steps.
\end{abstract}

\section{Introduction}
The Newton method is one of the most fundamental algorithms in optimization. Its origins trace back to the works of \citet{Newton} and \citet{Raphson}, and its numerous variants have found applications across mathematics, statistics, and machine learning. For example, the survey of trust-region and quasi-Newton methods by \citet{conn2000trust} cites over a thousand related papers. The method seeks a minimizer of a smooth twice differentiable function $f:\R^d \to \R$ via the update rule
\begin{equation}
x^{k+1}=x^k-\ls \h(x^k)\rs^{-1}\g(x^k).
\end{equation}
It is widely appreciated for its extremely fast local convergence and for its affine invariance, which ensures robustness to scaling and changes of coordinate basis. 

Despite its rich history, even the most basic variants of the Newton methods -- stepsize schedules -- are still being researched to this day.
\citet{nesterov1994interior} proposed one of the first stepsize schedules with global convergence guarantees. More recently, \citet{hanzely2022damped} established a duality between Newton stepsize schedules and Levenberg–Marquardt regularization, achieving a global $\mathcal{O}(1/k^2)$ rate. \citet{hanzely2024newton} extended this analysis under assumptions similar to third-order tensor methods, obtaining a simple schedule with $\mathcal{O}(1/k^3)$ convergence.

When the loss function has unknown parameterization, explicit stepsize schedules are often replaced by linesearch procedures, including greedy rules or backtracking based on standard descent conditions such as Frank–Wolfe, strong Frank–Wolfe, Armijo, or Goldstein conditions. For Newton direction, these procedures typically select stepsizes from the range $(0,1]$ \citep{nocedal2006numerical}.  

All of these methods, however, fundamentally rely on convexity. In fact, without convexity the direction $-\ls \h(x^k)\rs^{-1}\g(x^k)$ is not necessary a descent direction, and the stepsized Newton methods can diverge, both in theory and in practice.
Moreover, explicit stepsize rules \citep{nesterov1994interior, hanzely2022damped, hanzely2024newton} rely on the \emph{Newton decrement} $\la \g(x^k), \ls \h(x^k)\rs^{-1}\g(x^k) \ra$, which is not necessary positive outside convexity. 
One more nail to the coffin of nonconvex analysis is viewpoint that the Newton linesearches can be interpreted as Hessian-regularized updates \citep{hanzely2024newton}:
\begin{equation}
    x^{k+1}=x^k-\ls \h(x^k)+\lambda_k\mI \rs^{-1}\g(x^k),
\end{equation}
with $\lambda_k \propto \|x^{k+1}-x^k\|^\beta$ or $\lambda_k \propto \|\nabla f(x^k)\|^\beta$. In convex settings, this regularization admits a natural interpretation, since $\lambda_k$ can be related to the Newton decrement, yielding an equivalent update along the Newton direction. However, for nonconvex functions, the analogous regularization based on Newton decrement would be negative, making the analogy invalid. 
This raises a question:  
\begin{quote}
\centering
\emph{Is convexity truly a necessary condition for the fast convergence of Newton’s method, or can one guarantee such convergence in the nonconvex regime?}
\end{quote}
In this work we show that the standard form of convexity is not required for the stepsized Newton method to converge. We prove that Newton’s stepsize schedules are \emph{invariant under loss transformations}, enabling convexification of objectives without altering the iterates themselves. This new property, which we call \emph{transformation invariance}, extends the known geometric invariances of Newton’s method.

\subsection{Summary of the Contributions}

\begin{enumerate}
\item \textbf{Conceptual advance:} We introduce the notion of \emph{transformation invariance} and prove that the stepsized Newton method enjoys this property (\Cref{th:transformation_to_stepsize}). In contrast, closely related Hessian regularization techniques do not (\Cref{le:regularization_not_loss_invariant}).

\item \textbf{Theoretical consequences:} Transformation invariance enables convexification (\Cref{th:convexification_general}) and star-convexification (\Cref{th:star_convex}) of pseudoconvex losses. We provide sufficient and necessary conditions for such transformations to exist (\Cref{th:convexification_general}).

\item \textbf{Practical consequences:} We propose a \emph{transformation-induced stepsize schedule} (\Cref{col:induced_stepsize}), which transfers the iterate sequence of the Newton method applied to a transformed loss back to the original objective.

\item \textbf{Explaining non-standard stepsizes:} Our framework provides a principled explanation for the effectiveness of unconventional stepsize ranges in Newton’s method, including stepsizes larger than one (\Cref{ssec:example_poly}, \Cref{ssec:example_polytope}) and even negative stepsizes (Figures~\ref{fig:flip_poly}, \ref{fig:flip_log}).
\end{enumerate}

\subsection{Preliminaries}
We consider the unconstrained minimization problem 
\begin{align}
\min_{x\in \R^d} f(x),
\end{align}
where $f:\R^d \to \R$ is a twice-differentiable, nonconvex function, and $\xopt \eqdef \argmin_{x \in \R^d} f(x)$ denotes its global minimizer. We denote $\|\cdot\|$ for the standard Euclidean norm, and we will also be using norms induced by symmetric positive semidefinite matrices $\mathbf B \in \R^{d \times d}$ and their pseudoinverses,
\begin{eqnarray*}
\|h\|_{\textbf{B}}^2 \eqdef \<h,\textbf{B}h>,
\qquad \|g\|_{\textbf{B}}^{*2} = \<h,\textbf{B}^{\dag}h>,
\qquad \forall h,g\in \R^d.
\end{eqnarray*}
In particular, we frequently use local Hessian norms, $\mathbf B = \nabla^2 f(x)$, for which we use the shorthand
\begin{eqnarray*}
\|h\|_{x}^2\eqdef \<h,\nabla^2 f(x)h>,
\qquad \|g\|_x^{*2} \eqdef \<g,\nabla^2f(x)^{\dag}g>.
\end{eqnarray*}

Throughout the paper, we analyze properties of convex and pseudoconvex functions.
\begin{definition}[Convexity]
Function $f:\R^d \rightarrow \R$ is called convex, if
\[f(y) \geq f(x) + \<\nabla f(x),y-x>,~~~\forall x,y \in\R^d.\]
\end{definition}
We relax the standard convexity assumption to the following notion of pseudoconvexity.
\begin{definition}[Pseudoconvexity]\label{def:pseudo}
Function $f:\R^d \rightarrow \R$ is called pseudoconvex, if
\[f(y) < f(x) \Rightarrow \<\nabla f(x), y-x> < 0,~~~\forall x,y \in \R^d.\]
Additionally, $f$ is called strictly pseudoconvex, if
\[f(y) \leq f(x)\Rightarrow \<\nabla f(x), y-x> < 0,~~~\forall x\neq y \in \R^d.\]
\end{definition}
This formulation highlights the close connection between convexity and pseudoconvexity. For our analysis, we rely on an equivalent characterization from \citet{avriel1978second}.
\begin{lemma}\label{le:pseudoconvex}
Function $f$ is (strictly) pseudoconvex if and only if the following conditions hold:
\begin{enumerate}
\item $v^T\nabla f(x) = 0 \Rightarrow v^T\nabla^2 f(x)v \geq 0$,
\item If $\nabla f(x) = 0$, then $x$ is a (strict) global minimum.
\end{enumerate}
\end{lemma}

\section{Loss Transformations}
In this section, we analyze loss transformations. Consider a mapping $\phi:\R \to \R$ and the composite function $L \eqdef \phi \circ f$. Instead of minimizing the original, difficult nonconvex objective $f$, we may seek a transformation $\phi$\footnote{The transformation $\phi$ must be monotonically increasing on $[f(\xopt),\infty)$ to preserve the minimizer, i.e., $\argmin_{x \in \R^d} L(x) = \xopt$.} such that the transformed loss $L = \phi \circ f$ is convex and easier to optimize.
By the chain rule, the gradient and Hessian of $L$ are
\begin{align}
    \gloss (x) &= \gtrans(f(x)) \g(x),\label{eq:l_grad}\\
    \hloss (x) &= \gtrans (f(x)) \h(x) + \htrans(f(x)) \g(x) \g(x)^\top. \label{eq:l_hess}
\end{align}
Rearranging, we observe that $\hloss (x) \propto \h(x) + \frac {\htrans(f(x))}{\gtrans (f(x))} \g(x) \g(x)^\top,$ which suggests that $\loss$ can became convex if $\trans$ is chosen appropriately.
A natural candidate is the exponential transformation $\phi(f(x)) = \exp(a \cdot f(x))$, in which case the Hessian simplifies to $\hloss (x) \propto \h(x) + a \g(x) \g(x)^\top$.\\

Thus, a carefully selected transformation can substantially improve the optimization landscape. Once convexified, the transformed objective $L$ can be efficiently minimized using the Newton method,
\begin{align}
    \xnext
    = x-\atrans \ls \hloss (x) \rs^{-1} \gloss (x).
\end{align}

\subsection{Equivalency of the Newton's method}

Surprisingly, applying the stepsized Newton method to the original loss $f$ or to its transformed version $L = \phi \circ f$ yields almost identical behavior. The descent directions coincide, and the only difference lies in the stepsize. This observation is formalized in the following theorem.
\begin{theorem}\label{th:transformation_to_stepsize}
Let $\alpha(x)$ be an arbitrary Newton stepsize schedule for the original loss $f(x)$. If $\nabla f(x) \in \text{Range}\left(\nabla^2 f(x)\right)$, then the Newton method on the transformed loss $L(x)$ with stepsize
\begin{align}
\atrans (x) \eqdef \alpha (x) \lr 1+ \frac {\ \htrans(f(x))} {\gtrans(f(x))} \normsMd {\g(x)} x \rr
\end{align}
produces the same sequence of iterates as the Newton method on $f(x)$ with stepsize $\alpha(x)$, i.e.,
\begin{align}
    x-\alpha(x) \ls \h(x)\rs^\dag \g(x)
    = x-\atrans(x) \ls \hloss (x) \rs^{\dag} \gloss (x).
\end{align}
We refer to this property as the \emph{transformation invariance} of the stepsized Newton method. The multiplicative term $1+ \frac {\ \htrans(f(x))} {\gtrans(f(x))} \normsMd {\g(x)} x$ is called the \emph{scaling factor}. 
\end{theorem}
Therefore, if we have a stepsize schedule $\atrans$ for the Newton method minimizing a nice function $L$, we can induce a schedule for the difficult nonconvex function $f$ with matching sequences of iterate.
\begin{corollary} \label{col:induced_stepsize}
    A Newton method on the original function $f$ with stepsize schedule $\alpha(x)\eqdef \atrans(x) \lr 1+ \frac {\ \htrans(f(x))} {\gtrans(f(x))} \normsMd {\g(x)} x \rr^{-1}$ produces the identical sequence of iterates to the Newton method on the transformed loss $L$ with stepsize schedule $\atrans(x)$.
\end{corollary}
We illustrate the importance of this result in \Cref{fig:newton-trajectories_1D} (see also \Cref{ssec:transformation-induced_stpesize} for details).  
\begin{figure}[tbh]  
  \centering
  \includegraphics[width=\linewidth]{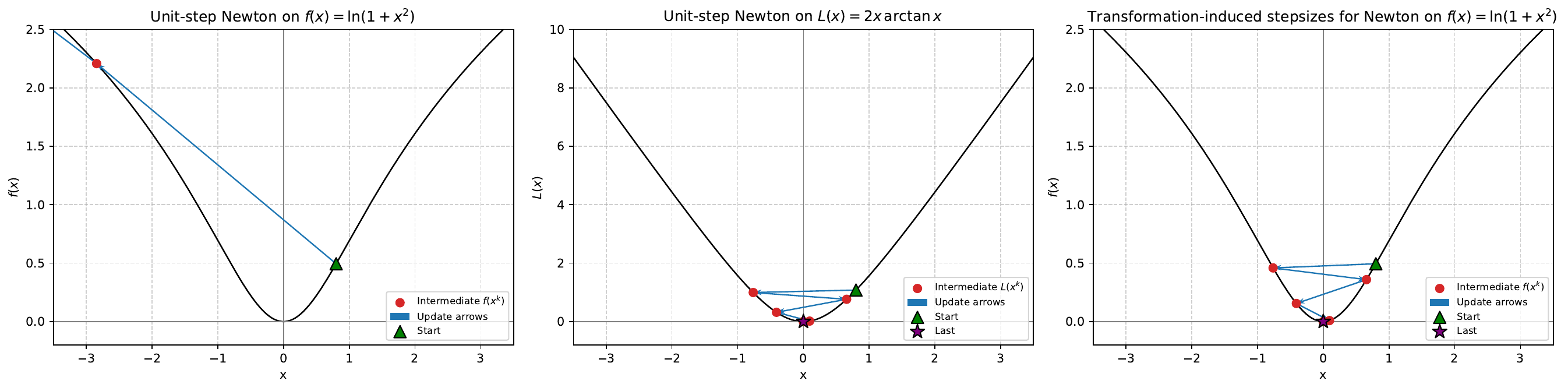}
  \caption{\textbf{Toy example: effect of transformation-induced stepsizes.}  
  Loss $f(x) = \ln(1+x^2)$ initialized at $x_0 = 0.8$.  
  Left: the classical Newton method diverges.  
  Middle: Newton method on the surrogate $L(x) = 2x \arctan(x)$ converges to the minimizer.  
  Right: Newton method on the original loss with a transformation-induced stepsize schedule also converges.}
  \label{fig:newton-trajectories_1D}
\end{figure}

Despite the known connection between Newton stepsizes and Levenberg--Marquardt Hessian regularization \citep{hanzely2022damped}, transformation invariance does not extend to the standard Levenberg--Marquardt scheme. This limitation is formalized below.

\begin{lemma}\label{le:regularization_not_loss_invariant}
Consider the sequence of iterates generated by the Newton method with a Levenberg--Marquardt regularization schedule $\lambda(x)$ applied to $f(x)$. 
There exists \textbf{no} regularization schedule $\lambda_{\phi}(x,\phi(\cdot), \lambda(x))$ for $\phi(f(x))$ producing the identical sequence of iterates.
\end{lemma}

Next, we turn to the geometric properties of the transformation $\phi$ and provide motivating examples.

To preserve the minimizer, we restrict $\phi$ to be monotonically increasing on the range of $f(x)$. Since $\phi$ preserves contour lines, a necessary condition for $\phi \circ f$ to be convex is that the sublevel sets of $f$ are themselves convex.

The stepsize scaling factor $1+ \frac {\ \htrans(f(x))} {\gtrans(f(x))} \normsMd {\g(x)} x$  is invariant under rescaling or translation of $\phi$: for linear transformations of the loss, the scaling factor is exactly one. A larger scaling factor implies that the transformed loss admits larger stepsizes, while a negative factor allows negative stepsizes (\Cref{th:transformation_to_stepsize}). Importantly, in all cases the sequences of iterates remain identical. Thus, if the stepsized Newton method converges on the original loss, the transformation-induced schedule guarantees convergence as well.

Having established these basic geometric properties of $\phi$, we now turn to concrete motivating examples.

\subsection*{Example: Polynomial Function} \label{ssec:example_poly}
It is well known that the classical Newton method minimizes a quadratic function in a single step. A natural question is whether this property extends to higher-order polynomials. \Cref{th:transformation_to_stepsize} suggests that the answer is indeed positive. For a positive definite matrix $\mA \succ \mathbf 0$ and an exponent $p \neq 1$, denote
\begin{align}
f(x) = \frac 1p \normM x \mA ^p,
\qquad \g(x) = \normM x \mA ^{p-2} \mA x,
\qquad \h(x) = (p-2) \normM x \mA ^{p-4} \mA x x^\top \mA + \normM x \mA ^{p-2}\mA.
\end{align}
Running Newton method on $f$ yields update
\begin{align}
\xnext &= x-\alpha \hi \g(x)\\
&=x-\alpha \ls (p-2) \normM x \mA ^{p-4} \mA x x^\top \mA + \normM x \mA ^{p-2}\mA \rs ^{-1} \normM x \mA ^{p-2}\mA x\\
&= \lr \mI - \alpha \ls (p-2) \normM x \mA ^{-2} x x^\top \mA + \mI  \rs^{-1} \rr x,\\
\intertext{and because $\ls (p-2) \normM x \mA ^{-2} x x^\top \mA + \mI  \rs x = (p-1)x$, so $\ls (p-2) \normM x \mA ^{-2} x x^\top \mA + \mI  \rs^{-1} x = \frac 1{p-1}x$ for $p\neq 1$, we can conclude}
&= \lr 1-\frac \alpha {p-1} \rr x.
\end{align}
Therefore, stepsize $\alpha = p-1$ guarantees convergence in 1 iteration.

\subsection{Example: Polytope Feasibility} \label{ssec:example_polytope}
Polytope feasibility problem searches for a point from a polytope $\Big\{ x \in \R^d| \langle a_i, x \rangle \leq b_i, \; \forall i \in \{1, \dots, n\} \Big\} $, and can be reformulated with $(t)_+ \eqdef \max\{t, 0\}$ and $p \geq 2$ as
\begin{equation}\label{eq:polytope}
    \min_{x \in \R^d }\Big\{ f_p(x) \eqdef \sum_{i=1}^{n} (\langle a_i, x \rangle - b_i)^p_+ \Big\}.
\end{equation}
\citet{hanzely2022damped} observed that for this problem the optimal fixed stepsizes for the Newton method are approximately $0.95, 1.95, 2.95, 3.95$ for $p=2,3,4,5$, respectively.  
Our theory provides a natural explanation for this phenomenon via the approximation $f_p(x) \approx \lr f_2(x) \rr^{p/2}$.
\begin{table}
    \centering
    \renewcommand{\arraystretch}{1.7}
    \begin{threeparttable}[h]{
            \caption{Comparison of loss transformations and corresponding stepsize compensation factor.}
            \label{tab:transformation_comparison}
            \centering 
            \begin{tabular}{c c c}
            \Xhline{4\arrayrulewidth}
            Loss transformation & Loss transformation formula $\phi(f_x)$ &Stepsize scaling factor \\
            \Xhline{4\arrayrulewidth}
             Linear & $a\cdot f_x + b$ & $1$\\
             \hline
             Polynomial  & $f_x^r$ & $1+ \frac {(r-1)}{f_x}\normsMd{\g(x)}x$ \\
             \hline
             Exponential & $\exp(a \cdot f_x)$ & $1+a \normsMd{\g(x)}x$\\
             \hline 
             Logarithmic & $\log(a+f_x)$ & $1- \frac 1{a+f_x} \normsMd {\g(x)} x$\\
             \hline
             Sigmoid & $\left(1 + \exp(-f_x)\right)^{-1}$ & $1 + \left(1 - 2\sigma(f_x)\right)\normsMd{\g(x)}x$\\
             \Xhline{4\arrayrulewidth}
            \end{tabular}
        }
    \end{threeparttable}
\end{table}

    

    

\section{Transformations Improving Objective Properties}

The Newton method and its variants are well known to perform reliably on convex problems, achieving superior convergence compared to the nonconvex case. However, if a nonconvex function can be convexified through an appropriate transformation, then the guarantees of \Cref{col:induced_stepsize} can be leveraged. In this section, we analyze function classes that admit such convexifying transformations and propose several concrete examples.  
(The detailed proofs are deferred to the Appendix.)

We begin with repeating the simple observation: a monotonically increasing transformation $\phi$ preserves both level sets and sublevel sets.
\begin{claim}
    If the sublevel sets of $f$ are nonconvex, then $\phi \circ f$ is nonconvex for any monotone $\phi$.
\end{claim}
While the convexity of sublevel sets is a standard property of convex functions, it is also enjoyed by pseudoconvex functions. 
\begin{claim} 
If $f$ is pseudoconvex, then its sublevels sets are convex.
\end{claim}
Another crucial property for convexification is that the gradient should vanish only at the solution. In fact, even one-dimensional functions with vanishing gradients away from the minimizer may fail to be convexifiable.
\begin{example}
Consider function $f(x) = |1 + (x-1)^5|$. Although its sublevel sets are convex, the function cannot be convexified via any monotone transformation. Notably, $f$ is not pseudoconvex.
\end{example}
Fortunately, pseudoconvex functions avoid this pathology: their gradients do not vanish outside the solution (\Cref{le:pseudoconvex}).  
In fact, strict pseudoconvexity is sufficient to guarantee the existence of a convexifying transformation, as we elaborate next.

\subsection{Convexification}
Let us analyze the Hessian of the transformed function $\loss=f \circ \trans$ under the monotone mapping $\phi$,
\[\hloss (x) = \gtrans (f(x)) \h(x) + \htrans(f(x)) \g(x) \g(x)^\top.\]
Since $\phi'(f(x)) > 0$ by monotonicity, we can equivalently write
\[\hloss (x) \propto \h(x) + \frac {\htrans(f(x))}{\gtrans (f(x))} \g(x) \g(x)^\top,\]
and denoting $r(x)=\frac {\htrans(f(x))}{\gtrans (f(x))}$,
\begin{equation} \label{eq:Hessian_L}
\hloss (x) \propto \h(x) +r(x)\g(x) \g(x)^\top.    
\end{equation}
Thus, convexification reduces to finding a transformation $\phi$ such that $\nabla^2 L(x)$ is positive semidefinite. For any $v \in \R^d$, $v^\top \hloss(x) v\propto v^\top \h(x) v + r(x) \lr v^\top \g(x)\rr^2 $. Convexity therefore requires that $v^\top \h(x) v>0$ for all vectors perpendicular to the gradient, i.e., $v^T\nabla f(x) = 0$. This is precisely the first condition of pseudoconvexity stated in \Cref{le:pseudoconvex}. 
The next step is to analyze the admissible choices of $r(x)$. To bound $r(x)$ from below, we employ the notion of the \emph{bordered Hessian}.
\begin{definition}
We call the bordered Hessian of twice differentiable loss $f: \R^d \to \R$ the matrix
\[
B(x) \;=\;
\begin{pmatrix}
0 & \nabla f(x)^{T} \\[6pt]
\nabla f(x) & \nabla^2 f(x)
\end{pmatrix}.
\]
We denote $D_{i_1,\ldots,i_k}(x)$ the principal minor of $B(x)$ of size $k+1$, formed by rows $0,i_1,\ldots, i_k$.

Analogously, we denote $M_{i_1,\ldots,i_k}(x)$ the principal minor of size $k$ of $\nabla^2 f(x)$, formed by rows $i_1,\ldots, i_k$, with the shorthand notation $M_k(x)$ for the leading principal minor of size $k$.
\end{definition}

With the introduced notation we can describe the sufficient conditions on $r(x)$ for convexification.

\begin{theorem}[\cite{schaible1980convexifiability}]\label{th:r_formula}
If $f$ is strictly pseudoconvex, then consider
\[r(x) = \begin{cases}
\max\{0; -1/\nabla f(x)^T\nabla^2f(x)\nabla f(x) : M_n(x) < 0\} & \text{for strictly pseudoconvex $f,$}\\
\max_{i_1,\ldots, i_k}\{0; M_{i_1,\ldots,i_k}(x)/D_{i_1,\ldots,i_k}(x) : D_{i_1,\ldots,i_k} < 0\} & \text{for pseudoconvex $f$}
\end{cases} 
\]
Then, $\h(x) + r(x)\g(x)\g(x)^\top$ is positive semidefinite.
\end{theorem}
\Cref{th:r_formula} establishes the existence of a sufficient $r(x)$ to convexify any pseudoconvex function $f$ at a given point $x$. However, since these formulas depend on $x$ through $\nabla f(x)$ and $\nabla^2 f(x)$, they do not directly yield a closed-form global transformation $\phi$. For a global result, we need $r$ to depend only on functional values $f(x)$.  
We address this by finding a global bound $h(f(x))$ such that
\[h(f(x)) \geq r(x),\]  
which can then be used to derive a valid transformation $\phi$, as stated below.
\begin{theorem}\label{th:convexification_general}
Let there be a global upper bound of $r(x)$ in terms of functional value, $h(f(x)) \geq r(x), \, \forall x\in \R^d$. Then for any monotonically increasing mapping $\trans: [f(\xopt), \infty) \to \R$ such that
\[\phi(y) \geq \exp\left(\int\limits_{f(x_*)}^{y}\exp\left(\int\limits_{f(x_*)}^w h(s)ds\right)dw\right)\]
makes the function $\trans \circ f$ convex.
\end{theorem}
In the theorem above, the bound is global, based on properties of $f$ over the entire unconstrained domain. However, if the algorithm satisfies monotone convergence, $f(x^{k+1}) \leq f(x^k)$ 
(as is typical for most second-order methods), then the iterates remain within the compact sublevel set $\mathcal{L}_{f, f(x^0)} \eqdef \{x| f(x) \leq f(x^0)\}$. In this case, properties outside the compact set are irrelevant, and the global bound $h$ can be improved by restricting it to $\mathcal{L}_{f, f(x^0)}$. For instance, the distance from the initial point to the solution, $\|x^0 - \xopt\|$, can be replaced by the diameter of this compact set.

Similarly, for Newton stepsize schedules that monotonically decrease the function value, it is sufficient to consider convexification restricted to a compact set $\Delta \eqdef \mathcal{L}_{f, f(x^0)}$.
\begin{corollary}
Let $f:\R^d \to \R$ be a pseudoconvex function (\Cref{def:pseudo}) with $r(x)$ bounded by a constant on a compact $\Delta$, $c\eqdef \argmax_{x\in \Delta} r(x) < \infty$.
Then any monotonically increasing mapping $\trans: [f(\xopt), \infty) \to \R$ such that
\[\phi(y) \geq \frac{\exp\left(c(y-f(x_*))\right)-1}{c}\]
makes the function $ \trans \circ f$ convex on the compact $\Delta$.
\end{corollary}
For practical implementations, one can simply take $L(x) = e^{c(f(x)-f_*)}$ for sufficiently large $c$, which ensures $\nabla^2f(x) + c\nabla f(x)\nabla f(x)^T$ positive semidefinite. 
is positive semidefinite. Under this exponential transformation, the stepsize multiplier is easily computed and the convergence properties from the convex setup are directly transferred to the pseudoconvex function.

\subsection{Star-Convexification}

While convexity is a powerful and widely used assumption, in some cases weaker conditions are sufficient to guarantee convergence. In particular, star-convexity is sufficient for the well-known Cubic Newton Method \citep{nesterov2006cubic}. This motivates the construction of \emph{star-convexifying transformations}.  

To derive convexity in one dimensional setup, one should compute maximum gradient norm over a subset. Without closed analytical expression, this might not be possible, since the loss landscape can be intricate, and its examining is another investigated problem. It turns out, that we can star-convexify the loss with much the closed form transformation.

\begin{theorem}\label{th:star_convex}
Let $f:\R^d \to \R$ be a strictly pseudoconvex (\Cref{def:pseudo}) function. Then $g$ defined as below is convex,
\begin{equation}
g(x) = f(x_*) + \int\limits_{0}^1 \frac{\<\nabla f(x_* + t(x-x_*)),x-x_*>}{t}dt.
\end{equation}
\end{theorem}
This is a much simpler form, as one does not have to inspect the gradient's norm and is able to calculate the function' value explicitly. However, if the loss function is radial symmetric and pseudoconvex, this definition is correct, as it depends only on the function values and is a valid loss transformation.
\begin{corollary}
Let $g$ be defined as in Theorem \ref{th:star_convex}. If $f(x) = \psi\left(\left\|x-x_*\right\|\right)$, then 
\begin{equation}
g(x) = f(x_*) + \psi^{-1}\left(f(x)\right)\int\limits_{f(x_*)}^{f(x)}\frac{dv}{\psi^{-1}\left(v\right)}.
\end{equation}
\end{corollary}

\newcommand{\fx}{f(x)}
\begin{table}
    \centering
    \setlength\tabcolsep{3pt} 
    \renewcommand{\arraystretch}{1.7}
    \begin{threeparttable}[h]
        {\caption{Examples of radial symmetric functions of the form $f(x) = \psi(\|x-x_*\|)$ that are nonconvex, but star-convex after the transformation from \Cref{col:radial_star_convexification}, \eqref{eq:radial_L}.}
            \label{tab:radial_symmetric}
            \centering 
            \begin{tabular}{c c c c}
            \Xhline{4\arrayrulewidth}
            Loss & \makecell{Formula\\ $\psi(x)=$} & \makecell{Transformed loss\\$L(x)=$} & \makecell{Transformation \\ $\trans(\fx)=$} \\
            \Xhline{4\arrayrulewidth}
             \makecell{Geman-\\-McClure \citenum{geman1986bayesian}}& $\frac{x^2}{x^2+1}$ & $\frac{x^2}{x^2+1} + x\arctan (x)$ & $\fx + \sqrt{\frac{\fx}{1-\fx}}\arctan\left(\sqrt{\frac{\fx}{1-\fx}}\right)$\\
             Welsh \citenum{dennis1978techniques} & $1-e^{-x^2}$ & $\sqrt{\pi}x\, \textnormal{erf}(x)$ & $\sqrt{-\pi\log\left(1-\fx\right)}\textnormal{erf}\left(\sqrt{-\log\left(1-\fx\right)}\right)$ \\
             Cauchy \citenum{black1996robust}& $\log(1+x^2)$ & $2x\arctan(x)$ & $2\sqrt{e^{\fx} - 1}\arctan\left(\sqrt{e^{\fx} - 1}\right)$\\
             \hline
             \Xhline{4\arrayrulewidth}
            \end{tabular}
        }
    \end{threeparttable}
\end{table}

Since even in in one-dimensional case star-convexity does not imply convexity \citep{nesterov2006cubic}, one should carefully check the limitations of used optimizers. However, if derivatives of the initial one-dimensional function $\psi$ behave well far from the optimum, then, we might expect $g$ to be convex, and the Newton Method to converge.\\

\begin{corollary}
If $\psi''(r) + \frac{\psi'(r)}{r} \geq 0$ for $r \in [0,M]$, then, $g$ is convex on $\mathcal{N} = \left\{x~|~\|x-x_*\| \leq M \right\}$.
\end{corollary}

\begin{table}[H]
    \centering
    \setlength\tabcolsep{6pt} 
    \renewcommand{\arraystretch}{1.7}
    \begin{threeparttable}[h]
        {\caption{Change of the neighborhood of the convergence for the unitary stepsize Newton method for losses examples from \Cref{tab:radial_symmetric}. All loses are minimized at the origin.}
            \label{tab:star-convex-neighborhoods}
            \centering 
            \begin{tabular}{c c c c c}
            \Xhline{4\arrayrulewidth}
            Loss & \makecell{Formula\\ $\psi(x)=$} & \makecell{Original \\ convergence radius} & \makecell{Transformed loss\\$L(x)=$} & \makecell{Transformed \\ convergence radius} \\
            \Xhline{4\arrayrulewidth}
             \makecell{Geman-\\-McClure \citenum{geman1986bayesian}}& $\frac{x^2}{x^2+1}$ & $\frac 1{\sqrt 3} \approx 0.577$ & $\frac{x^2}{x^2+1} + x\arctan (x)$ & $1$\\
             Welsh \citenum{dennis1978techniques} & $1-e^{-x^2}$ & $\frac 1{\sqrt 2} \approx 0.707$ & $\sqrt{\pi}x\, \textnormal{erf}(x)$ & 1 \\
             Cauchy \citenum{black1996robust}& $\log(1+x^2)$ & $1$ & $2x\arctan(x)$ & $\infty$\\
             \hline
             \Xhline{4\arrayrulewidth}
            \end{tabular}
        }
    \end{threeparttable}
\end{table}
\section{Numerical Experiments}

\subsection{Reversal of the Descent Direction} \label{ssec:sign_flip}
Our theory predicts that the scaling factor can become negative for certain losses (see \Cref{tab:transformation_comparison}), leading to a negative stepsize and hence a reversal of the descent direction. We verify this phenomenon in practice.  

\Cref{fig:flip_poly} (and \Cref{fig:flip_log} in the Appendix) confirm that such sign flips indeed occur. The regions in which the stepsize factor becomes negative depend on both the loss and the chosen transformation. For polynomial transformations with very small $r$, large portions of the domain yield negative scaling factors. For logarithmic transformations, negative scaling factors appear over large regions regardless of the parameter $a$.

\begin{figure}[H] 
  \centering
  \includegraphics[width=0.95\linewidth]{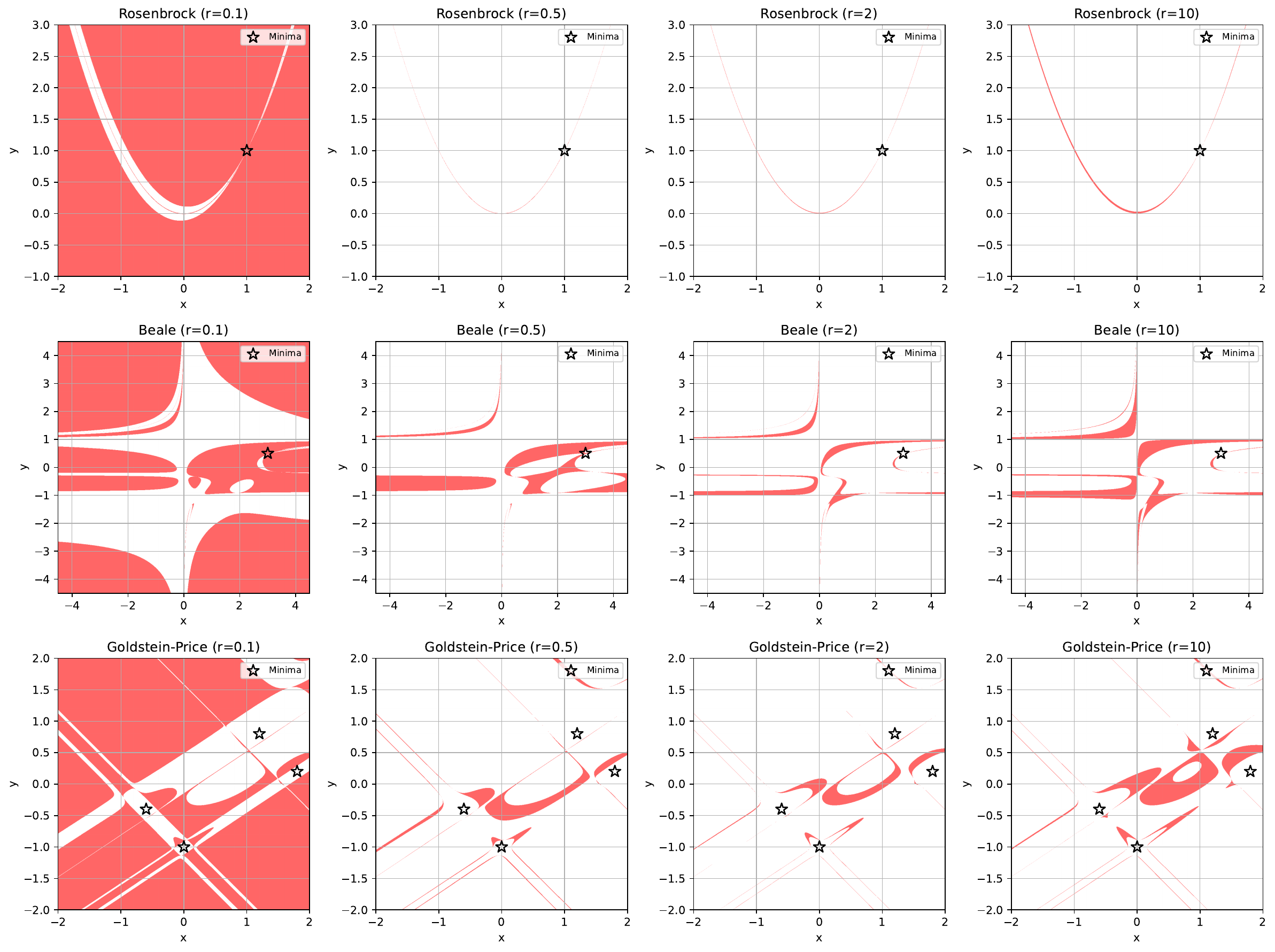}
  \caption{Regions where the Newton step changes sign after a polynomial loss transformation.}
  \label{fig:flip_poly}
\end{figure}

\subsection{Convergence Neighborhood}
Our theory also suggest that transformations can alter the convergence neighborhood of the classical Newton method. We verify this numerically:  
\Cref{fig:convergenc_neighborhood} demonstrates how polynomial transformations change the regions of convergence.
\begin{figure}[H] 
  \centering
  \includegraphics[width=0.95\linewidth]{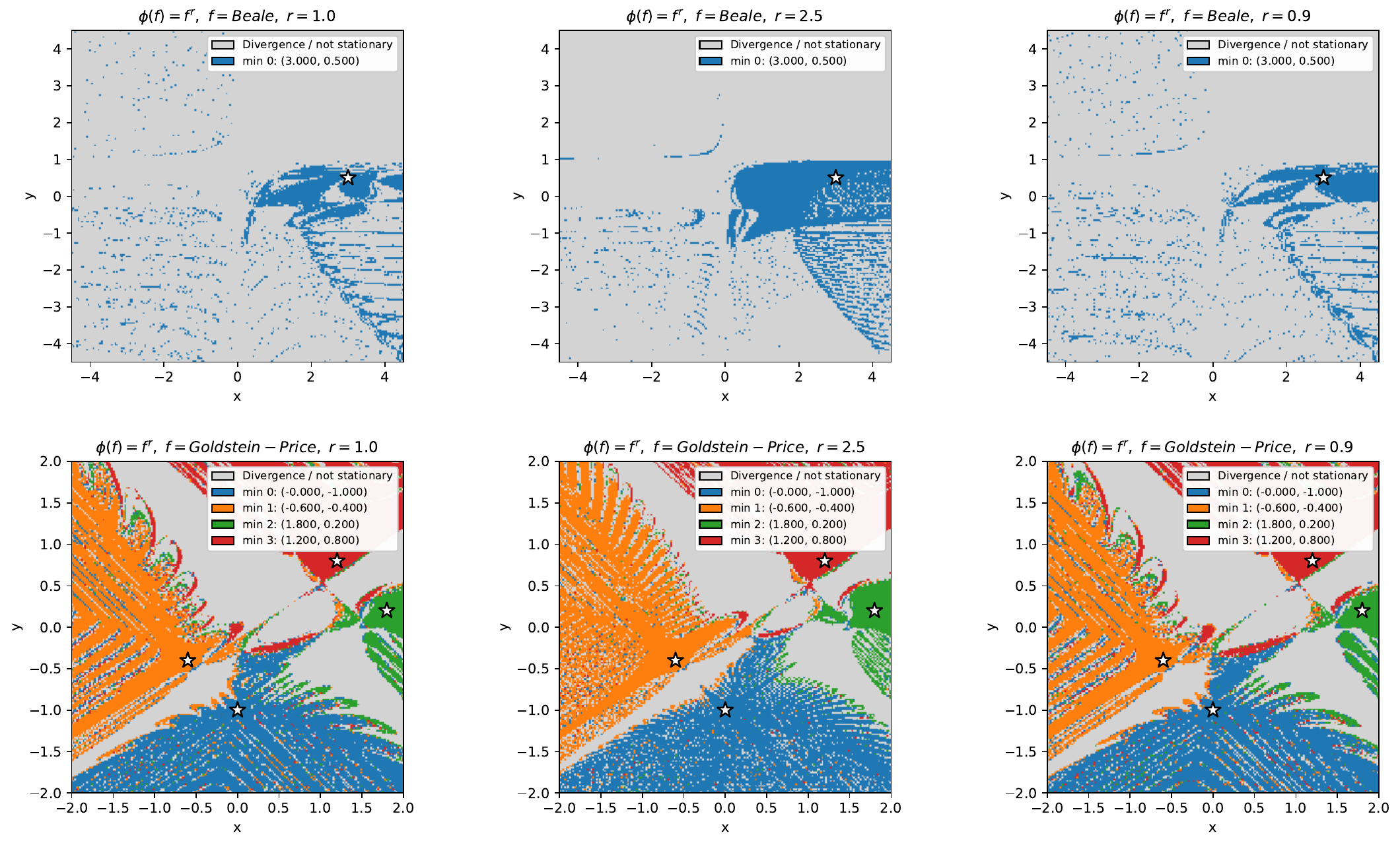}
  \caption{Convergence regions under polynomial transformations $\phi(f)=f^r$ on Beale and Goldstein--Price functions.}
  \label{fig:convergenc_neighborhood}
\end{figure}

\subsection{Recovering Convergence by Stepsize Rescheduling} \label{ssec:transformation-induced_stpesize}
We demonstrate a practical application of transformation-induced stepsizes: recovering convergence in cases where the standard Newton method diverges.  

Consider the one-dimensional function $f(x) = \ln(1+x^2)$, which has minimizer at $x^* = 0$. If the initialization satisfies $|x_0| \geq 1/\sqrt{3}$, the unit-step Newton method diverges. For example, with $x_0 = 0.8$, the iterates diverge (left subplot of \Cref{fig:newton-trajectories_1D}).  

If we instead transform $f$ into $L(x) = 2 \arctan(x)$ (as in \Cref{tab:radial_symmetric}) and apply the unit-step Newton method to $L$, the iterates converge to $x^*$ (middle subplot). Moreover, the sequence of iterates on $L$ can be transferred back to the original loss $f$ using the transformation-induced stepsize schedule from \Cref{col:induced_stepsize}, restoring convergence (right subplot).

\subsection{Benchmark Losses} \label{ssec:losses}
We evaluate our approach on three benchmark objectives, each defined as a mapping $f:\R^2 \to \R$. These functions capture different challenges in optimization: the Rosenbrock loss features narrow curved valleys, the Beale loss exhibits strong nonlinearity, and the Goldstein--Price loss is multimodal. The formulas for these objectives are given by
\begin{align}
f_{\mathrm{Rosenbrock}}(x,y) &= (1 - x)^2 + 100\,(y - x^2)^2, \\
f_{\mathrm{Beale}}(x,y) &= (1.5 - x + xy)^2 + (2.25 - x + xy^2)^2 + (2.625 - x + xy^3)^2,\\
f_{\mathrm{Goldstein\text{-}Price}}(x,y) 
&= \bigl[1 + (x+y+1)^2 \,(19 - 14x + 3x^2 - 14y + 6xy + 3y^2 )\bigr] \nonumber \\
&\quad \times \bigl[30 + (2x-3y)^2 \,(18 - 32x + 12x^2 + 48y - 36xy + 27y^2 )\bigr].
\end{align}


\bibliography{references}
\bibliographystyle{iclr2026_conference}
\newpage
\appendix

\section{Additional Experiments}
In this section we present and visualisation of the regions with the negative scaling factors for the logarithmic loss transformation in
\Cref{fig:flip_log}.
\begin{figure}[tbh] 
  \centering
  \includegraphics[width=0.95\linewidth]{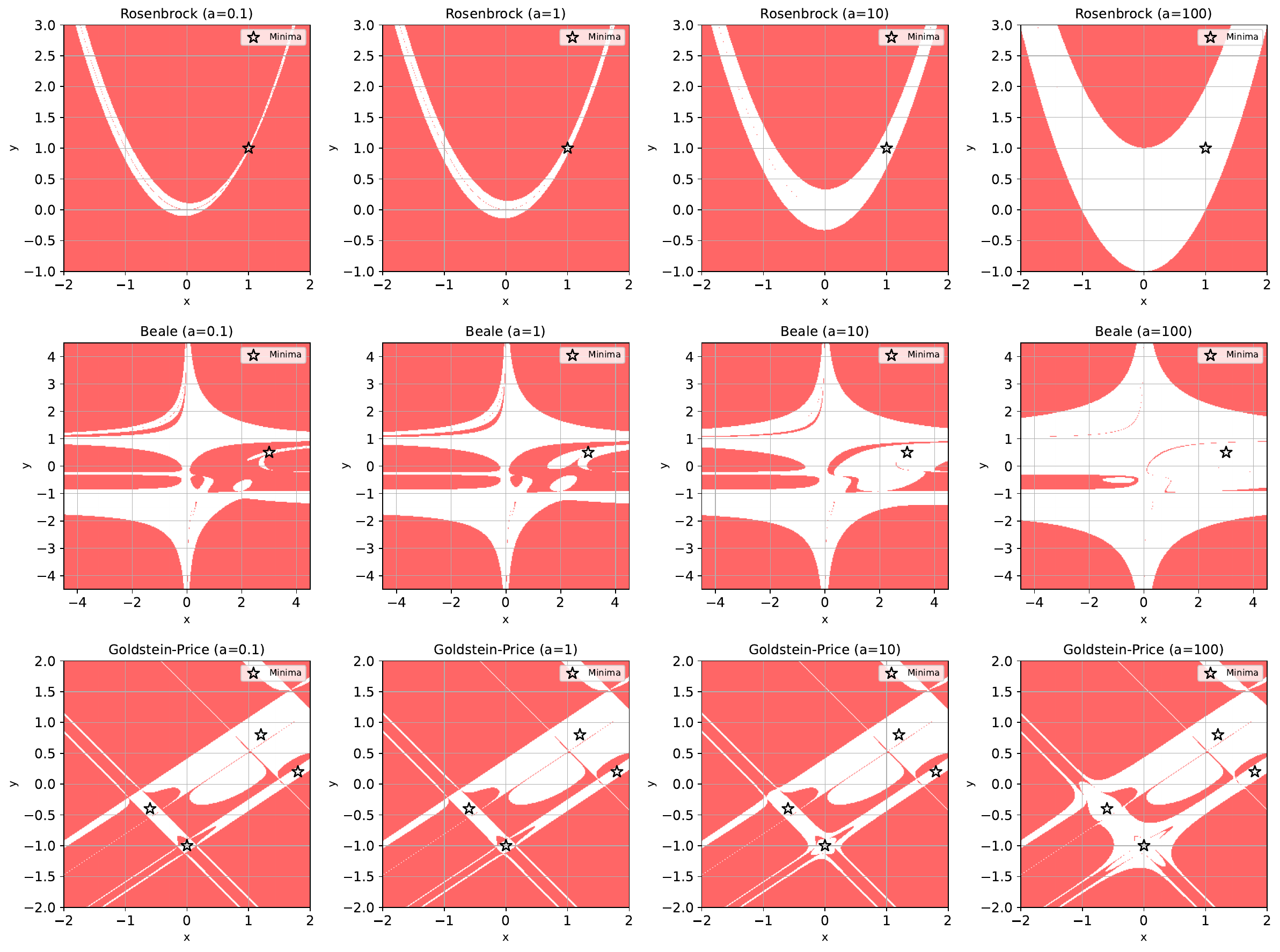}
  \caption{Regions where the Newton step changes sign after a logarithmic loss transformation.}
  \label{fig:flip_log}
\end{figure}

\section{Further Connections to the Literature}

\subsection{Newton Method Stepsize Range}
For Newton and quasi-Newton methods, virtually all explicit stepsize rules and linesearches restrict stepsizes to the range $(0,1]$. 
However, the optimality of this range has recently been challenged by \citet{shea2024dont}, who show that allowing negative stepsizes can, in many cases, be more effective than restricting to positive ones only.  

Our results provide a theoretical justification for this observation: under transformation invariance, the induced stepsize can naturally be larger than one or even negative. This offers a principled explanation for the effectiveness of such unconventional choices.  

\subsection{Newton Method}
While first-order methods often achieve satisfactory performance, incorporating second-order information can dramatically accelerate convergence. The Newton method, in particular, enjoys superlinear convergence near the optimum and performs well on convex problems, often significantly outperforming gradient descent. Nonetheless, the classical Newton method has several limitations, which have motivated numerous extensions.  

Globalization techniques include adding Hessian regularization terms \citep{nesterov2006cubic}, employing stepsize schedules \citep{polyak2009regularized,hanzely2022damped}, and using Levenberg--Marquardt regularization \citep{levenberg1944method, doikov2024gradient}. When computing the exact Hessian is prohibitively expensive, algorithms with inexact Hessians are used. Quasi-Newton methods \citep{davidon1991variable, Fletcher1988PracticalMO} approximate the Hessian action on the gradient using heuristic updates \citep{nocedal2006numerical} or rank constraints \citep{ye2021explicit}. Other approaches reuse the Hessian across multiple iterations to reduce computational cost \citep{doikov2023second}.  

Preconditioning methods can also be interpreted as inexact Hessian techniques, ranging from simple diagonal approximations \citep{singh2020woodfisherefficientsecondorderapproximation} to more sophisticated variants related to natural gradient descent \citep{Amari1998NaturalGW}. Accelerated schemes for convex problems have also been developed for convex objectives \citep{nesterov2008accelerating, agafonov2023advancing}. Although Newton-type methods are increasingly being explored in nonconvex settings \citep{doikov2024spectral}, their convergence rates and guarantees are generally weaker than in the convex case.  

\subsection{Pseudoconvex Losses}
Convex optimization \citep{tyrrell1970convex} has long served as the cornerstone of optimization theory, thanks to the absence of local minima and the availability of efficient algorithms. Classical methods include gradient-based techniques \citep{polyak1963gradient}, higher-order methods \citep{nesterov2021implementable}, and algorithms for constrained optimization \citep{frank1956algorithm}. Even today, convex optimization results continue to shed light on the learning dynamics of modern machine learning models \citep{schaipp2025surprising}.  

With the rise of deep learning, however, many practical loss landscapes are nonconvex \citep{cooper2021global}, motivating analysis under weaker assumptions. This has led to the development of generalized convexity frameworks \citep{cambini2009generalized} and the study of local properties.  

Quasiconvex functions \citep{greenberg1971review} extend convexity by requiring convex sublevel sets, enabling broader algorithmic applicability. Pseudoconvex functions \citep{mangasarian1975pseudo, crouzeix2020generalized} go further: they retain convex sublevel sets and exclude local minima, thereby preserving key advantages of convexity. \citet{avriel1978second} provided characterizations for differentiable and twice-differentiable functions in these classes.

Pseudoconvex losses have proven particularly useful in robust optimization \citep{beyer2007robust,barron2019general}. They behave like convex functions near the optimum—ensuring uniqueness of the minimizer—yet their reduced sensitivity to large deviations makes them more robust to noise. This sacrifices strict convexity while maintaining pseudoconvexity.
\section{Preservation of Newton Method's Iterations}

\subsection{Proof of Theorem \ref{th:transformation_to_stepsize}}
\label{ssec:proof_transformation_to_stepsize}

\begin{proof}[\pof{\Cref{th:transformation_to_stepsize}}]
\textbf{Notation.}
Throughout this proof let
\[
\gr \eqdef \g(x),\quad \mH \eqdef \h(x),\quad
\gtrans \eqdef \gtrans(f(x)),\quad
\htrans \eqdef \htrans(f(x)).
\]
For the transformed loss $L(x) = \phi(f(x))$ we have
\[
\gloss(x)=\gtrans\,\gr,\qquad
\hloss(x)=\htrans\,\gr\gr^\top+\gtrans\,\mH.
\]
Assume $\nabla f(x_k)\in\text{Range}(\nabla^2 f(x_k))$; abbreviate $x=x_k$ and $x_{k+1}=\xnext$.

\medskip
Starting from the transformed step and substituting $\gloss,\hloss$,
\begin{align*}
    \xnext - x
    &= -\,\atrans(x)\,\Bigl(\hloss(x)\Bigr)^{\dag}\,\gloss(x)\\
    &= -\,\atrans\,\Bigl(\htrans\,\gr\gr^\top \mH^{\dag} + \gtrans\,\mI\Bigr)^{\dag}\,\gtrans\,\mH^{\dag}\,\gr.
\end{align*}
Note that
\[
\Bigl(\htrans\,\gr\gr^\top \mH^{\dag} + \gtrans\,\mI\Bigr)\,\gr
= \bigl(\htrans\,\|\gr\|^2_{\mH^{\dag}} + \gtrans\bigr)\,\gr,
\quad\text{since}\quad (\gr\gr^\top \mH^{\dag})\gr=\gr\,(\gr^\top \mH^{\dag}\gr).
\]
Hence
\[
\Bigl(\htrans\,\gr\gr^\top \mH^{\dag} + \gtrans\,\mI\Bigr)^{\dag}\,\gr
= \frac{1}{\htrans\,\|\gr\|^2_{\mH^{\dag}} + \gtrans}\,\gr,
\]
and therefore
\begin{align*}
    \xnext - x
    &= -\,\atrans\,\frac{\gtrans}{\htrans\,\|\gr\|^{*2}_{x} + \gtrans}\,\mH^{\dag}\,\gr
     \;=\; -\,\alpha\,\mH^{\dag}\,\gr,
\end{align*}
with the scalar stepsize
\[
\alpha \;=\; \atrans\,\frac{\gtrans}{\htrans\,\|\gr\|^{*2}_{x} + \gtrans}.
\]
This is exactly the claimed form.
\end{proof}
\subsection{Proof of Lemma \ref{le:regularization_not_loss_invariant}}
It can be shown, that even with invertible matrices, Levenberg-Marquardt regularization cannot be considered transformation invariant.
\begin{proof}[\pof{\Cref{le:regularization_not_loss_invariant}}]
Assume, for contradiction, that an LM regularization on $f$ is transformation‑invariant.
Let
\[
\gr \eqdef \g(x),\quad \mH \eqdef \h(x),\quad
\gtx\eqdef \gtrans(f(x)),\quad \htx\eqdef \htrans(f(x)),\quad
\ltrans \eqdef \ltrans(x,\phi(\cdot),\lambda(x)).
\]
Suppose the (damped) Newton step satisfies
\begin{eqnarray*}
    -\bigl(x^{+}-x\bigr)
    \;&=&\; \bigl(\mH+\lambda(x)\mI\bigr)^{-1}\gr
    \;=\; \bigl(\hloss(x)+\ltrans\mI\bigr)^{-1}\gloss(x)\\
    \;&=&\; \Bigl(\htx\,\gr\gr^\top + \gtx\,\mH + \ltrans\,\mI\Bigr)^{-1}\,\gtx\,\gr.
\end{eqnarray*}
Equivalently,
\[
\bigl(\mH+\lambda(x)\mI\bigr)^{-1}\gr
= \Bigl(\mH + \tfrac{\htx}{\gtx}\,\gr\gr^\top + \tfrac{\ltrans}{\gtx}\,\mI\Bigr)^{-1}\gr.
\]
For generic $(\mH,\gr)$ this identity imposes $d$ scalar equations on the single unknown $\ltrans$ (the rank‑one term depends on $\gr$), so it admits no solution in general. Thus a scalar LM coefficient cannot be chosen to make the regularized step invariant under $\phi\circ f$.
\end{proof}

\section{Examples of Non-Convexifiable Losses}
\begin{example}[Nonconvex sublevel sets]\label{ex:nonconvex_sublevels}
Let $\mathcal{L}_{f,c}\coloneqq\{x:\; f(x)\le c\}$.
Assume there exist $x,y\in\mathcal{L}_{f,c}$ and $t\in(0,1)$ such that $z\coloneqq tx+(1-t)y\notin \mathcal{L}_{f,c}$, i.e.,
\[
f(z)>c\ge \max\{f(x),f(y)\}.
\]
For any monotone $\phi$, we have $x,y\in\mathcal{L}_{\phi\circ f,\max\{\phi(f(x)),\phi(f(y))\}}$, while by monotonicity,
\[
z\notin\mathcal{L}_{\phi\circ f,\max\{\phi(f(x)),\phi(f(y))\}}.
\]
Hence $\phi\circ f$ has a nonconvex sublevel set and therefore cannot be convex.
\end{example}

\begin{example}[Nonconvexifiability of $\,f(x)=|1+(x-1)^5|$]
Let $x_1=1$ and $x_0\coloneqq 1-2^{1/5}$.
Then $f(x_1)=f(x_0)=1$, but $f'(x_1)=0$ while $f'(x_0)\ne 0$.
Suppose there exists a monotone $\phi$ with $\phi(0)=0$, $\phi(1)=1$ such that $L=\phi\circ f$ is convex.
For any $\varepsilon>0$, since $1=\frac{\varepsilon\cdot 0 + 1\cdot(1+\varepsilon)}{1+\varepsilon}$, convexity gives
\[
L(1)\le \frac{\varepsilon}{1+\varepsilon}L(0)+\frac{1}{1+\varepsilon}L(1+\varepsilon)
\;\;\Rightarrow\;\;
1\le \frac{L(1+\varepsilon)-L(1)}{\varepsilon}.
\]
Now $L(1+\varepsilon)=\phi\!\bigl(1+\varepsilon^5\bigr)$, hence
\[
\frac{1}{\varepsilon^4}
\;\le\;
\frac{\phi(1+\varepsilon^5)-\phi(1)}{\varepsilon^5}.
\]
Letting $\varepsilon\downarrow 0$, we obtain the right derivative $\phi'_+(1)=+\infty$.
By the chain rule where $f$ is differentiable,
$L'(x)=\phi'(f(x))\,f'(x)$, so at $x_0$ we would have $|L'(x_0)|=+\infty$ (since $f(x_0)=1$ and $f'(x_0)\ne 0$), which is impossible for a finite‑valued convex function on $\R$.
Therefore no monotone $\phi$ can convexify $f$.
\end{example}

\section{Missing Proofs for (Star-)Convexification}
\begin{lemma}\label{lem:star-conv}
Let $f$ be pseudoconvex with unique minimizer and twice differentiable. Then,
\[L(x) = f(x_*) + \int\limits_0^1\frac{\langle \nabla f\left(x_* + t(x-x_*)\right),x-x_*\rangle}{t}dt\]
is star-convex.
\end{lemma}
\begin{proof}
First of all, we show, that this definition is correct, i.e. $L$ exists. For twice differentiable $f$ we have
\begin{eqnarray*}
\nabla f(x + u) = \nabla f(x) + \nabla^2 f(x)u + o\left(u\right),~~~u\rightarrow 0.
\end{eqnarray*}
Therefore, 
\begin{eqnarray*}\frac{\<\nabla f(x_* + t(x-x_*))x-x_*,>}{t} &=& \frac{t\<\nabla^2 f(x_*)(x-x_*),x-x_*> + o(t)}{t} \\
&=& \<\nabla^2 f(x_*)(x-x_*),x-x_*> + o(1),~~~ t\rightarrow 0.
\end{eqnarray*}
Hence, integral converges. Then, we shall proof the star-convexity of $L$.  For $\lambda \in [0,1]$ we have
\begin{eqnarray*}
L\left((1-\lambda)x_* + \lambda x\right) &=& L\left(x_* + \lambda(x-x_*)\right) \\
&=& f(x_*) + \int\limits_{0}^1 \frac{\<\nabla f(x_* + t\lambda(x-x_*)),\lambda (x-x_*)>}{t}dt\\
&=& f(x_*) + \lambda \int\limits_{0}^{\lambda}\frac{\<\nabla f(x_* + s(x-x_*)),(x-x_*)>}{s}ds
\end{eqnarray*}
As $f(x_* + s(x-x_*)) > f(x_*)$, we have
\begin{eqnarray*}
0 < \<\nabla f(x_* + s(x-x_*)),s(x_*-x) > = s \<\nabla f(x_* + s(x-x_*)),x_*-x>.
\end{eqnarray*}
Hence,
\begin{eqnarray*}
L\left(x_* + (1-\lambda)x\right) &=& f(x_*) + \lambda \int\limits_{0}^{\lambda}\frac{\<\nabla f(x_* + s(x-x_*)),(x-x_*)>}{s}ds\\
&\leq& f(x_*) + \lambda \int\limits_{0}^{1}\frac{\<\nabla f(x_* + s(x-x_*)),(x-x_*)>}{s}ds \\
&=& (1-\lambda)L(x_*)+ \lambda L(x).
\end{eqnarray*}
\end{proof}
\begin{corollary}\label{col:radial_star_convexification}
For radial symmetric functions $f(x) = \psi(\|x-x_*\|)$ the transformed loss 
\begin{equation}\label{eq:radial_L}
L(x) = f(x_*) + \|x-x_*\| \int\limits_{0}^{\|x-x_*\|}\frac{\psi'(t)}{t}dt = f(x_*) + \psi^{-1}(f(x))\int\limits_{f(x_*)}^{f(x)}\frac{dv}{\psi^{-1}(v)}
\end{equation}
\end{corollary}
\begin{proof}
If $f(x)  = \psi(\|x-x_*\|)$, then, $\nabla f(x) = \psi'(\|x-x_*\|)\frac{x-x_*}{\|x-x_*}$. Therefore, $L$ can be rewritten as following:
\begin{eqnarray*}
L(x) = f(x_*) + \int\limits_{0}^1 \frac{\|x-x_*\|\psi'(t\|x-x_*\|)}{t}dt = f(x_*) + \|x-x_*\|\int\limits_{0}^{\|x-x_*\|}\frac{\psi'(s)}{s}ds.
\end{eqnarray*}
If $v = \psi(s)$, then, $dv = \psi'(s)ds$. Therefore, $\frac{\psi'(s)}{s}ds = \frac{dv}{s} = \frac{dv}{\psi^{-1}(v)}$, and we can rewrite
\begin{eqnarray*}
L(x) = f(x_*) + \psi^{-1}(f(x))\int\limits_{f(x_*)}^{f(x)}\frac{dv}{\psi^{-1}(v)}.
\end{eqnarray*}
\end{proof}
\begin{lemma}
    If $\psi''(r) + \frac{\psi'(r)}{r} \geq 0$ for $r \in [0,M]$, then, $L$ is convex on $\mathcal{N} = \left\{x~|~\|x-x_*\| \leq M \right\}$.
\end{lemma}
\begin{proof}
    Define $\Psi(r) = r\int\limits_{0}^r \frac{\psi'(t)}{t}dt$. Then, $\Psi'(r) = \int\limits_{0}^r\frac{\psi'(t)}{t}dt + \psi'(r)$ and $\Psi''(r) = \frac{\psi'(r)}{r} + \psi''(r)$.

    Then, we obtain $\nabla L(x) = \Psi'(\|x-x_*\|)\frac{x-x_*}{\|x-x_*\|}$ and $\nabla^2 L(x) = \Psi''(\|x-x_*\|)\frac{(x-x_*)(x-x_*)^T}{\|x-x_*\|^2} + \frac{\Psi'(r)}{r}\left(\mI - \frac{(x-x_*)(x-x_*)^T}{\|x-x_*\|^2}is \right)$. Therefore, there are 2 eigenvalues: $\lambda_1 = \Psi''(r)$ and $\frac{\Psi'(r)}{r}$, which we need to examine.

    $\Psi''(r) = \psi''(r) + \frac{\psi'(r)}{r} \geq 0$, as we assumed.

    Using $\psi'(0) = 0$, we obtain $\Psi'(r) = \int\limits_{0}^r\frac{\psi'(t)}{t}dt + \psi'(r) = \int\limits_{0}^r\left(\frac{\psi'(t)}{t}+\psi''(t)\right)dt \geq 0$. Therefore, we obtain the needed.
\end{proof}
For losses from \Cref{tab:radial_symmetric}, we can derive first and second derivatives,
\[
\phi_1'(c) = 1 + \frac{1}{2 \sqrt{c (1-c)^3}} \int_0^c \sqrt{\frac{1-v}{v}} \, dv
\]

\[
\phi_1''(c) = \frac{-1 + 2c}{4 \left(c (1-c)^3\right)^{3/2}} \int_0^c \sqrt{\frac{1-v}{v}} \, dv + \frac{1}{2 \sqrt{c (1-c)^3}}
\]

\[
\phi_2'(c) = 1 + \frac{e^c}{2 \sqrt{e^c - 1}} \int_0^c \frac{dv}{\sqrt{e^v - 1}}
\]

\[
\phi_2''(c) = \frac{e^c (e^c - 2)}{4 (e^c - 1)^{3/2}} \int_0^c \frac{dv}{\sqrt{e^v - 1}} + \frac{e^c}{2 \sqrt{e^c - 1}}
\]

\[
\phi_3'(c) = 1 + \frac{1}{2 (1-c) \sqrt{-\log(1-c)}} \int_0^c \frac{dv}{\sqrt{-\log(1-v)}}
\]

\[
\phi_3''(c) = \frac{2 + \log(1-c)}{4 (1-c)^2 (-\log(1-c))^{3/2}} \int_0^c \frac{dv}{\sqrt{-\log(1-v)}} + \frac{1}{2 (1-c) \sqrt{-\log(1-c)}}
\]
\begin{lemma} 
Let $f$ be pseudoconvex and $\phi$ be monotone. Then, 
\[L(x) = \exp(\phi\left(f(x)\right)\]
will be convex for 
$$\phi(y) \geq \exp\left(\int_{f(x_*)}^y \exp\left(\int_{f(x_*)}^w r(s)ds\right)dw\right)$$
\end{lemma}
\begin{proof}
If $\phi$ is strictly monotone, then $L$ is convex iff $\nabla^2 f(x) + \frac{\phi''(f(x))}{\phi'(f(x))}\nabla f(x)\nabla f(x)^T$ is positive (semi) definite. Assume, that there is $r(x) \geq 0$, such that $\nabla^2 f(x) + r(f(x))\nabla f(x)\nabla f(x)^T$ is positive (semi) definite. Then, with $g(y) =\phi'(y), y_* = f(x_*), \phi(y_*) = 0, \phi'(y_*) = 1$ (w.l.o.g.) we have 
\begin{eqnarray*}
g'(y) &\geq& g(y)r(y)\\
g(y) &\geq& g(y_*)\exp\left(\int_{y_*}^y r(s)ds\right)\\
\phi'(y) &\geq& \exp\left(\int_{y_*}^y r(s)ds\right)\\
\phi(y) &\geq& \exp\left(\int_{y_*}^y \exp\left(\int_{y_*}^w r(s)ds\right)dw\right).
\end{eqnarray*}
Criteria for $r$ can be found in \cite{schaible1980convexifiability}. However, frequently, $r$ depends on $x$ and not on $f(x)$.This can be fixed on a compact set: $r \eqdef \max_{x\in \Delta}r(x)$. Therefore,
\begin{eqnarray*}
\phi(y) &=& \exp\left(\int_{y_*}^y \exp\left(\int_{y_*}^w r ds\right)dw\right) = \exp\left(\int_{y_*}^y \exp (r(w-y_*))dw\right) =\\
 &=& \frac{1}{r}\left(e^{r(y-y_*)} - 1\right)
\end{eqnarray*}
is sufficient for convexification on $\Delta$
\end{proof}


\end{document}

